\documentclass[12pt]{amsart}

\usepackage{amssymb,amsfonts}
\usepackage[all,arc]{xy}
\usepackage{enumerate}
\usepackage{mathrsfs}
\usepackage{tikz-cd}
\usepackage{hyphenat}
\usepackage{geometry}

\newcommand\restr[2]
{{
  \left.\kern-\nulldelimiterspace 
  #1 
  \vphantom{\big|} 
  \right|_{#2} 
  }}
  \DeclareMathOperator{\dom}{dom}
  \DeclareMathOperator{\ord}{ord}
\newtheorem{thm}{Theorem}[section]
\newtheorem{cor}[thm]{Corollary}
\newtheorem{prop}[thm]{Proposition}
\newtheorem{lem}[thm]{Lemma}

\theoremstyle{definition}
\newtheorem{defn}[thm]{Definition}

\newtheorem{exmp}[thm]{Example}

\theoremstyle{remark}
\newtheorem{rem}[thm]{Remark}
\newtheorem{rems}[thm]{Remarks}

\makeatletter
\let\c@equation\c@thm
\makeatother
\numberwithin{equation}{section}

\newgeometry{tmargin=3.5cm, bmargin=4cm, lmargin=3.6cm, rmargin=3.6cm}

\title[Maximal ideals of regulous functions are not finitely generated]{Maximal ideals in the ring of regulous functions are not finitely generated}

\author{Aleksander Czarnecki}

\date{}
\subjclass{Primary 14P10; Secondary 26C15}
\begin{document}
\begin{abstract}

The paper consider regulous functions on the real affine space $\mathbb{R}^N$. We shall study some algebraic properties of the ring of those functions. It is presented a proof of the regulous version of Nullstellensatz based on the substitution property and the Artin-Lang property for the considered function ring. We prove that every maximal ideal in the ring of regulous functions on $\mathbb{R}^N$ when $N\geq 2$ is not finitely generated. Finally, we extend the latter result to an arbitrary, smooth, real affine algebraic variety of dimension $d\geq 2$.

\end{abstract}
\maketitle

\tableofcontents
\linespread{1.3} 
\section{Introduction}
The subject of our work are rational functions admitting a continuous and $\mathcal{C}^{k}$ extension to the real affine space $\mathbb{R}^N$, which are simply called $k\textrm{\hyp{}}$regulous. Section $2$ contains some basic results about $k\textrm{\hyp{}}$regulous functions and sets; for details we refer the reader to the papers [2] and [5]. We shall study some algebraic properties of the ring of those functions. Our work is focused on two purposes.\newline \indent First, in Section $3$, we present a proof of the regulous version of Nullstellensatz relying on the paper [3]. Our approach is based on the substitution property for the ring of regulous functions and on the Artin--Lang property which can be directly deduced from the former. Basically, these two results together with the formal Positivstellensatz allow us to establish Nullstellensatz for the ring of regulous functions. \newline \indent Next, in Section $4$, we will be concerned with lack of noetherianity of the ring  $\mathcal{R}^{k}\left( \mathbb{R}^{N}\right)$ of $k\textrm{\hyp{}}$regulous functions on $\mathbb{R}^N$, $N\geq 2$, which was proven in [2]. It is well known that there exist non-noetherian rings whose all maximal ideals are finitely generated. \newline \indent We shall show that every maximal ideal of $\mathcal{R}^{k}\left( \mathbb{R}^{N}\right)$, $N\geq 2$, is not finitely generated. Then we extend this result to smooth, real affine algebraic varieties of dimension $d\geq 2$. This is a new result which cannot be deduced from the reasonings in the paper [2]. Besides the methods we use are different from those in [2]. \newline \indent We assume some prerequisites of real algebraic geometry, for which we refer the reader to the monography [1]. In particular, we will use the Tarski--Seidenberg principle in many contexts.
\newline \newline \indent We use the following notation:
\newline $\mathbb{N}=\lbrace 0, 1, ...\rbrace$
\newline $\mathbb{N}^{\ast}=\lbrace 1, 2, ...\rbrace$
\newline $O=(0, ...,0) \in \mathbb{R}^{N}$
\newline $\mathcal{Z}(f)$ the zero set of a function $f$
\newline$\mathcal{Z}(E)$ the set of common zeroes of functions from a subset $E$
\newline $\mathcal{I}(F)$ the ideal of functions vanishing on a set $F$
\newline $\sqrt{I}$ the radical of an indeal $I$ 

\section{Regulous functions and their basic properties}
This introductory section presents without proofs some basic results about regulous functions. For details and proofs we refer the reader to [2], [4] and [5].
\newline \newline \indent Let $N\in \mathbb{N}^{\ast}$ be a positive integer, and $f\in \mathbb{R}(x_{1}, ..., x_{N})$ be a rational function on $\mathbb{R}^{N}$. The domain of $f$\ denoted by $\mathrm{dom}(f)$ is by definition the biggest Zariski open subset of $\mathbb{R}^{N}$ on which $f$ is regular. There exist $p,q\in \mathbb{R}[x_{1},...,x_{N}]$, with $q\neq 0$ on $\mathrm{dom}(f)$, such that $f=\dfrac{p}{q}$ on $\mathrm{dom}(f)$ (cf. [2,~Proposition~2.9]). The indeterminacy locus of $f$ is defined to be the Zariski closed set $\mathrm{indet}(f)=\mathbb{R}^{N}\setminus \mathrm{dom}(f)$.  
 \begin{defn}
 Let $N \in\mathbb{N}^{\ast}$, $k\in\mathbb{N}\cup \lbrace\infty\rbrace$. We say that a real valued function $f$ defined on $\mathbb{R}^{N}$ is $k\textrm{\hyp{}}$regulous if the following conditions are satisfied:
 \begin{enumerate}
 \item $f$ is rational on $\mathbb{R}^{N}$
 \item $f:\mathbb{R}^{N}\longrightarrow\mathbb{R}$ is of class $\mathcal{C}^{k}$. 
\end{enumerate} 
The set of $k\textrm{\hyp{}}$regulous functions on $\mathbb{R}^{N}$ has a natural ring structure. We denote this ring by $\mathcal{R}^{k}\left( \mathbb{R}^{N}\right)$.     
\end{defn}
\begin{rems}
Similarly, we can define the ring $\mathcal{R}^{k}\left( X \right)$ of $k\textrm{\hyp{}}$regulous functions on a smooth, real affine algebraic variety $X$ (cf. [2, Definition 2.15]). The following relations are worth noting: $$\mathcal{O}(X)=\mathcal{R}^{\infty}\left( X\right)\subseteq\cdots\subseteq\mathcal{R}^{1}\left( X\right)\subseteq\mathcal{R}^{0}\left( X\right),$$ where $\mathcal{O}(X)$ is the ring of regular functions on $X$ (cf. [2, Theorem 3.3]). We also say that a mapping $f:\mathbb{R}^N\rightarrow \mathbb{R}^M$ is $k\textrm{\hyp{}}$regulous when so are all its coordinate functions. 
\end{rems}

\begin{defn}
Let $f:\mathbb{R}^{N}\rightarrow\mathbb{R}$ be a real function and $k\in \mathbb{N}$ be an integer. We say that $f$ is $k\textrm{\hyp{}}$flat at a point $a\in\mathbb{R}^{N}$ if $f$ is of class $\mathcal{C}^{k}$ near $a$ and all partial derivatives of $f$ of order $\leq k$ are zero at $a$. We say that $f$ is $k\textrm{\hyp{}}$flat if $f$ is $k\textrm{\hyp{}}$flat at any point of its zero set. 
\end{defn}

\begin{exmp}
For $(k,N)\in \mathbb{N}\times\mathbb{N}^{*}$ define on $\mathbb{R}^{N}$ a real function $f=\dfrac{x_{1}^{3+k}}{x_{1}^{2}+...+x_{N}^{2}}$. Then $f$ is $k\textrm{\hyp{}}$regulous function on $\mathbb{R}^{N}$ and $k\textrm{\hyp{}}$flat at $O$. This can be checked by routine calculation. 
\end{exmp}
The next proposition says that the indeterminacy locus of a regulous function $f$ cannot be too wide.
\begin{prop}[[2, Proposition 3.5$\rbrack$] 
Let $f\in \mathcal{R}^{k}\left( \mathbb{R}^{N}\right)$, $f(x)=\dfrac{p(x)}{q(x)}$, for $x\in\mathrm{dom}(f)$ where $p$, $q$ are suitable relatively prime polynomials. Then $\mathcal{Z}(q)\subseteq\mathcal{Z}(p)$ and $\mathrm{codim}_{\mathbb{R}^{N}}\mathcal{Z}(q)\geq2$.
\end{prop}
Now we recall a crucial theorem on existence of regular stratification from which many important properties of regulous functions follow (cf. [2, Theorem 4.1], [5, Corollary 11.5] and [4, Proposition 8 ff.]).
\begin{thm}
Consider a $k\textrm{\hyp{}}$regulous function $f$ on $\mathbb{R}^{N}$. Then, there exist a finite stratification of $\mathbb{R}^{N}
 $ $$\mathbb{R}^{N}=\coprod_{i=1}^{m} S_{i}$$
where $S_{i}$ are Zariski locally closed subsets such that all restrictions $f\vert_{S_{i}}$ are regular functions.
\end{thm}
\begin{rem}
From the above theorem we can in particular deduce that the composition of $k\textrm{\hyp{}}$regulous functions remains $k\textrm{\hyp{}}$regulous. Namely, let $n,m,l\in\mathbb{N}^{\ast}$ and $f:\mathbb{R}^{n}\rightarrow\mathbb{R}^{m}$, $g:\mathbb{R}^{m}\rightarrow\mathbb{R}^{l}$ be $k\textrm{\hyp{}}$regulous functions. Then the function $g\circ f:\mathbb{R}^{n}\rightarrow\mathbb{R}^{l}$ is also $k\textrm{\hyp{}}$regulous (cf. [2, Corollary 4.14] and [5, Corollary 11.7]). 
\end{rem}
In the next section we will also need the following variant of the \L{}ojasiewicz property.
\begin{prop}[[2, Lemma 5.1$\rbrack$ and [5, Corollary 12.2 $\rbrack$]
Let $(k,N)\in \mathbb{N}\times\mathbb{N}^{*}$, $f, g\in \mathcal{R}^{k}\left( \mathbb{R}^{N}\right)$. Suppose that $\mathcal{Z}(g)\subseteq \mathcal{Z}(f)$. Then there exist natural $M$, and a $k\textrm{\hyp{}}$regulous function $h$, such that $$f^{M}=gh.$$
\end{prop}
Let us remind briefly the notion of $k\textrm{\hyp{}}$regulous topology. Just like in the construction of the Zariski topology, where we use as closed sets the sets of common zeroes of polynomials to obtain a well defined topology, we can do the same by replacing polynomials with $k\textrm{\hyp{}}$regulous functions to obtain again a well defined topology in the similar manner.      
\begin{defn}
 We say that subset $F\subseteq\mathbb{R}^{N}$ is $k\textrm{\hyp{}}$regulous closed if there exists a subset $E\subseteq\mathcal{R}^{k}\left( \mathbb{R}^{N}\right)$ such that:$$\mathcal{Z}(E)=F $$
We define a $k\textrm{\hyp{}}$regulous topology as a topology induced by family of $k\textrm{\hyp{}}$regulous closed sets.
\end{defn}

Recall that a topological space is noetherian if the closed subsets satisfy the descending chain condition: any descending chain of closed sets $$F_{1} \supseteq F_{2} \supseteq ... $$ stabilizes, i.e. $F_{m}=F_{m+1}=...$ for an integer $m$. \newline \newline 
 It can be proven, by means of regular stratification (Theorem 2.6), that the $k\textrm{\hyp{}}$regulous topology is noetherian (cf. [2, Theorem 4.3], [5, Proposition 11.10]). \newline \indent Actually even more is true. Recall that a subset $F \subseteq \mathbb{R}^{N}$ is constructible if it is a (finite) boolean combination of Zariski closed subsets of $\mathbb{R}^{N}$. Then the family of all closed and constructible subsets of $\mathbb{R}^{N}$ is the family of closed sets for a topology which we call the constructible topology on $\mathbb{R}^{N}$. Furthermore, the constructible topology is noetherian (cf. [5, Proposition 11.1]).  
It has been proven that for every $k\in \mathbb{N}$, the $k\textrm{\hyp{}}$regulous topology coincides with the constructible topology on $\mathbb{R}^N$ (cf.~[2,~Theorem 6.4]).   
 Hence, it is not necessary to specify the integer $k$ to define the regulous topology on $\mathbb{R}^N$. In particular it follows that the regulous topology, being exactly the constructible topology on $\mathbb{R}^N$ is automatically noetherian.  

\section{Nullstellensatz in the ring  $\displaystyle \mathcal{R}^{k}\left( \mathbb{R}^{N}\right)$ via Artin--Lang property}
In this section we present a proof of the regulous version of Nullstellensatz based on the paper [3]. We start with preliminaries from real algebra following the monography [1]. Later in this section we will also use the Tarski--Seidenberg principle, we refer the reader to [1,~chapters~1,~4 and 5].
\newline \newline \indent From now on $A$ will be an unitary commutative ring. An ideal $I$ in $A$ is said to be real if for any elements $a_{1}, ..., a_{n} \in A $ the following condition is satisfied: $a_{1}, ..., a_{n} \in I $ whenever $\displaystyle \sum_{i=1}^{n}a_{i}^2 \in I$. The real radical of the ideal $I$ is the set 

$$\sqrt[r]{I}=\left \{ a\in A \mid \exists m\in \mathbb{N},~\exists b_{1}, ..., b_{n} \in A,~a^{2m}+b_{1}^{2}+...+b_{n}^{2}\in I    \right \}.$$

\begin{rem}
By [1, Proposition 4.1.7], $\sqrt[r]{I}$ is the intersection of all real prime ideals of $A$ containing $I$. We also have  $I\subseteq\sqrt{I}\subseteq\sqrt[r]{I}$ with equality if and only if $I$ is a real (hence radical, cf. [1, Lemma 4.1.5]) ideal. 
\end{rem}
\begin{defn}
We say that a subset $\alpha \subseteq A$ is a cone if following conditions are satisfied:
\begin{enumerate}[i)]
\item $a,b\in \alpha\Rightarrow a+b\in \alpha$
\item $a,b\in \alpha\Rightarrow ab\in \alpha$
\item $a\in A\Rightarrow a^{2}\in \alpha$
\newline The cone $\alpha$ is said to be proper if in addition
\item $-1\notin \alpha$
\newline The proper cone is said to be prime if additionaly
\item $ab\in \alpha\Rightarrow$ $(a\in \alpha$ or $-b\in \alpha)$
\end{enumerate}
\end{defn}
The set of sums of squares of elements of $A$ is the smalest cone of $A$. We denote this set as $\sum A^{2}$. The positive cone of an ordered field $(F,\leq)$ is the cone $\alpha=\lbrace a\in F \mid a\geq 0\rbrace$.
\newline \indent Let $\alpha$ be a prime cone of $A$ and $-\alpha=\lbrace a\in A \mid -a\in \alpha  \rbrace$. Then $\alpha\cup-\alpha=A$ and $\alpha\cap-\alpha$ is a prime ideal of $A$, called the support $\mathrm{supp}(\alpha)$ of $A$. \newline \newline The following proposition gives a one to one correspondence between the orderings and the prime cones of a ring $A$.
\begin{prop}[[1, Proposition 4.3.4$\rbrack$]
A subset $\alpha \subseteq A$ is a prime cone of $A$ if and only if there exists an ordered field $(F,\leq)$ and a homomorphism $\varphi:A\rightarrow F$, such that $$\alpha=\lbrace a\in A \mid \varphi(a)\geq 0 \rbrace.$$
\end{prop}
\begin{lem}[[1, Lemma 4.3.5$\rbrack$]
For a prime cone $\alpha$ of $A$, denote by $k(\mathrm{supp}(\alpha))$ the residue field of A at $\mathrm{supp}(\alpha)$, i.e. the field of fractions of $A/\mathrm{supp}(\alpha)$. Then $$\overline{\alpha}=\left \{ \dfrac{\overline{a}}{\overline{b}}\in k(\mathrm{supp}(\alpha)) \mid ab \in \alpha \right \}$$ is the positive cone of an ordering of $k(\mathrm{supp}(\alpha))$, and $\alpha$ is the inverse image of $\overline{\alpha}$ under the canonical homomorphism $A \rightarrow k(\mathrm{supp}(\alpha))$. It follows, in particular, that $\mathrm{supp}(\alpha)$ is a real prime ideal.  
\end{lem}

\begin{thm}[[1, Theorem 4.3.7$\rbrack$]
The following conditions are equivalent:
\begin{enumerate}[i)]
\item The ring $A$ has a proper cone
\item The ring $A$ has a prime cone
\item There is a homomorphism $\varphi:A\rightarrow K$, where $K$ is a real closed field
\item The ring $A$ has a real prime ideal
\item The element $-1$ is not a sum of squares in $A$, i.e. $-1\notin \sum A^{2}$.

\end{enumerate}
\end{thm}
Now we show a one to one correspondence  between the orderings of residue field, the prime cones and the homomorphisms to a real closed field. This makes it possible to consider the real spectrum of a ring in the three equivalent ways, which will be useful later.   
\newline \newline \indent The prime cone $\alpha$ induces an ordering $\leq_{\alpha}$ of the residue field $k(\mathrm{supp}(\alpha))$. This ordering is defined by $0\leq_{\alpha} \overline{a}\Leftrightarrow a\in \alpha$, for every $a\in A$ (where $\overline{a}$ denotes the class of $a$ in $k(\mathrm{supp}(\alpha))$). We denote by $k(\alpha)$ the  real closure of the ordered field $(k(\mathrm{supp}(\alpha)), \leq_{\alpha})$. For $a\in A$, $a(\alpha)$ denotes the image of $a$ by the canonical homomorphism from $A$ into $k(\alpha)$: 
$$A\rightarrow A/\mathrm{supp}(\alpha)\rightarrow k(\mathrm{supp}(\alpha))\rightarrow k(\alpha).$$
If the ring $A$ is an $\mathbb{R}\textrm{\hyp{}}$algebra, then the above canonical homomorphism is an $\mathbb{R}\textrm{\hyp{}}$algebra homomorphism. Clearly we have:
\begin{enumerate}[i)]
\item $a(\alpha)\geq 0\Leftrightarrow a\in \alpha$
\item $a(\alpha)> 0\Leftrightarrow a\notin -\alpha$
\item $a(\alpha)= 0\Leftrightarrow a\in \mathrm{supp}(\alpha)$.
\end{enumerate}
We conclude with proposition given below:   
\begin{prop}[[1, Proposition 7.1.2$\rbrack$]
The following data are equivalent:
\begin{enumerate}[i)]
\item A prime cone $\alpha$ of $A$.
\item A couple $(\mathfrak{p},\leq)$ where $\mathfrak{p}$ is a prime ideal of $A$, and $\leq$ is an ordering of the residue field $k(\mathfrak{p})$. The ideal $\mathfrak{p}$ is then necessarily real $\mathrm{(cf.~[1,~Lemma~4.1.6])}$.
\item An equivalence class of homomorphisms $\varphi:A\rightarrow R$ with values in a real closed field, for the smallest equivalence relation, such that $\varphi$ and $\varphi'$ are equivalent if there exists a commutative diagram of homomorphisms:
$$\begin{tikzcd}[column sep=large]
A \arrow{r}{\varphi}  \arrow{rd}{\varphi'} 
  & R \arrow{d}{} \\
    & R'
\end{tikzcd}$$
\end{enumerate}
More precisely, one goes from i) to ii) by taking $(\mathfrak{p},\leq)=(\mathrm{supp}(\alpha), \leq_{\alpha})$, from ii) to iii) by taking $\varphi:A\rightarrow A/\mathfrak{p}\rightarrow k(\mathfrak{p})\rightarrow R$, where $R$ is the real closure of $k(\mathfrak{p})$ for $\leq$, and from iii) to i) by taking $\alpha=\lbrace a\in A \mid \varphi(a)\geq 0  \rbrace$. 
\end{prop}
Now we are ready to give a formal definition of the real spectrum of commutative ring $A$:
\begin{defn}
The real spectrum $\mathrm{Spec_{r}}A$ of $A$ is the topological space whose points are the prime cones of $A$, with topology given by the basis of open subsets of the form $$\widetilde{\mathit{U}}(a_{1},...,a_{n})=\lbrace             \alpha \in \mathrm{Spec_{r}}A \mid a_{1}(\alpha)>0,..., a_{n}(\alpha)>0  \rbrace,$$ where $a_{1},...,a_{n}$ is any finite family of elements of $A$. This topology is called the spectral topology.    
\end{defn}
The elements of $A$ are now regarded as "functions" on $\mathrm{Spec_{r}}A$.
\newline \newline \indent We still need some results listed below.
\begin{prop}[[3, Proposition 5.1$\rbrack$] 
Let $(k,N)\in\mathbb{N}\times\mathbb{N}^{\ast}$. Any boolean combination of subsets of the form $\lbrace x\in\mathbb{R}^{N} \mid f(x)>0\rbrace$, where $f\in\mathcal{R}^{k}\left( \mathbb{R}^{N}\right)$, is a semi-algebraic subset of $\mathbb{R}^{N}$.
\end{prop}
If $f$ is a rational function in $\mathbb{R}(x_{1},...,x_{N})$ and $\mathbb{R}\rightarrow R$ a real closed field extension, then one can define (independently of the representation as a quotient $f=\dfrac{p}{q},~p,q \in \mathbb{R}[x_{1},...,x_{N}])$ the function $f_{R}=\dfrac{p}{q}$ viewed as a rational function in $R(x_{1},..., x_{N})$. Let $S\subseteq \mathbb{R}^N$ be a semi-algebraic set given by a boolean combination $\mathcal{B}(x)$ of sign conditions on polynomials in $\mathbb{R}[x_{1}, ..., x_{N}]$, where $x=(x_{1}, ..., x_{N})$. The subset $\lbrace x\in R^{N} \mid \mathcal{B}(x) \rbrace$ of $R^{N}$, denoted by $S_{R}$, is semi-algebraic and depends only on the set $S$ and not on the boolean combination chosen to describe it (cf.~[1,~Proposition~5.1.1]). The set $S_{R}$ is called the extension of $S$ to $R$. We can also extend a function $f\in\mathcal{R}^{k}\left( \mathbb{R}^{N}\right)$. Indeed, by the Tarski--Seidenberg principle, $f$ as a rational function, has an unique extension to any real closed field containing  $\mathbb{R}$. We have the following:
\begin{prop}[[3, Proposition 5.2$\rbrack$]
Let $(k,N)\in\mathbb{N}\times\mathbb{N}^{\ast}$, $f\in\mathbb{R}(x_{1},...,x_{N})$ and $\mathbb{R}\rightarrow R$ a real closed field extension. Then $f\in\mathcal{R}^{k}\left( \mathbb{R}^{N}\right)$ if and only if $f_{R}\in\mathcal{R}^{k}\left( R^{N}\right)$.
\end{prop}
 Now we are going, following the paper [3], to present a proof of the regulous version of Nullstellensatz, which is based on the substitution property recalled below.
\begin{prop}
Let $(k,N)\in\mathbb{N}\times\mathbb{N}^{\ast}$. Any $\mathbb{R}\textrm{\hyp{}}$algebra homomorphism $$\phi:\mathcal{R}^{k}\left( \mathbb{R}^{N}\right)\rightarrow R,$$ where $R$ is a real closed extension of $\mathbb{R}$, is uniquely determined by its values on $\mathbb{R}[x_{1},x_{2},...,x_{N}]$, i.e. $\phi$ is an evaluation map.
\end{prop}
\begin{proof}
First, observe that $\restr{\phi}{\mathbb{R}[x_{1},x_{2},...,x_{N}]}$ is an evaluation map (as $\mathbb{R}\textrm{\hyp{}}$algebra homomorphism from the polynomial ring over $\mathbb{R}$ into a real closed extension $\mathbb{R}\subseteq R$). Take arbitrary $f\in\mathcal{R}^{k}\left( \mathbb{R}^{N}\right)$. We know by Theorem 2.6 that, there is a stratification on Zariski locally closed subsets $\mathbb{R}^{N}=S_{0}\cup...\cup S_{m}$ such that $f$ is regular on each $S_{i}$. From Tarski--Seidenberg's principle we also know that $R^{N}={S_{0}}_{R}\cup...\cup {S_{m}}_{R}$. Since $$\restr{\phi}{\mathbb{R}[x_{1},x_{2},...,x_{N}]} :\mathbb{R}[x_{1},x_{2},...,x_{N}]\rightarrow R$$ is an evaluation homomorphism, we find $x_{0}\in R^{N}$ such that $\phi(w)=w(x_{0})$, for all $w \in\mathbb{R}[x_{1},x_{2},...,x_{N}]$. Then changing numeration of indices we may assume that $x_{0}\in {S_{0}}_{R}$. Because $f$ is regular on $S_{0}$, we can write $f=\dfrac{p}{q}$ on $S_{0}$, where $p,q$ are polynomials in $\mathbb{R}[x_{1},x_{2},...,x_{N}]$, and $q$ is nonzero everywhere on $S_{0}$. Let $S_{0}= \lbrace x\in \mathbb{R}^{N} \mid r(x)=0, s(x)\neq 0\rbrace $, for some polynomials $r,s$. By Tarski--Seidenberg's principle ${S_{0}}_{R}= \lbrace x\in R^{N} \mid r(x)=0, s(x)\neq 0\rbrace $. Clearly, $\phi(r)=r(x_{0})=0$, $\phi(s)=s(x_{0})\neq0$, because $r,s \in \mathbb{R}[x_{1},x_{2},...,x_{N}]$. Again by Tarski--Seidenberg's principle ${S_{0}}_{R}\subseteq\lbrace x\in R^{N}\mid q(x)\neq 0 \rbrace$ and thus $\phi(q)=q(x_{0})\neq 0$. Observe that we have inclusion (in $\mathbb{R}^{N}$) of zero sets $\mathcal{Z}(r)\subseteq \mathcal{Z}(s\cdot(qf-p))$ (indeed, if $s(x)\neq 0$, then $ x\in S_{0}$ and we use regularity of $f=\dfrac{p}{q}$). Then by the \L{}ojasiewicz property (Proposition 2.8) there exist an integer $M$ and $g\in \mathcal{R}^{k}\left( \mathbb{R}^{N}\right)$ such that $(s\cdot(qf-p))^{M}=rg$. Applying $\phi$ we get $\phi(s)^{M}\cdot(\phi(q)\phi(f)-\phi(p))^{M}=\phi(r)\phi(g)$. Since $\phi(r)=0$ and $\phi(s)\neq 0$ we finally obtain that $\phi(q)\phi(f)-\phi(p)=0$ and we use the fact that $\phi(q)\neq 0$ to conclude that $\phi(f)=\dfrac{\phi(p)}{\phi(q)}$.   
\end{proof}
The proof of Nullstellensatz we present relies on the following Artin--Lang property for $\mathcal{R}^{k}\left( \mathbb{R}^{N}\right)$.
\begin{prop}
Let $(k,N)\in\mathbb{N}\times\mathbb{N}^{\ast}$. Let $f, f_{1},...,f_{r}$ in $\mathcal{R}^{k}\left( \mathbb{R}^{N}\right)$ and set $$S=\lbrace x\in \mathbb{R}^{N} \mid f(x)=0, f_{1}(x)>0,...,f_{r}(x)>0    \rbrace,$$ and $$\widetilde{S}=\lbrace \alpha\in \mathrm{Spec_{r}} \mathcal{R}^{k}\left( \mathbb{R}^{N}\right)\mid f(\alpha)=0, f_{1}(\alpha)>0,...,f_{r}(\alpha)>0    \rbrace.$$
Then $\widetilde{S}=\emptyset$ if and only if $S=\emptyset$.
\end{prop}
\begin{proof}
Since $S$ is a semi-algebraic subset of $\mathbb{R}^{N}$, there exist $(l,m)\in\mathbb{N}^{2}$ and some polynomials $p_{i,j}, q_{i}$, $1\leq i \leq l, 1\leq j \leq m$ in $\mathbb{R}[x_{1},x_{2},...,x_{N}]$, such that $$S=\bigcup_{i=1}^{l}S_{i}$$ with $S_{i}=\lbrace x\in \mathbb{R}^{N} \mid  q_{i}(x)=0, p_{i,1}(x)>0,...,p_{i,m}(x)>0    \rbrace$. Then for any real closed extension $\mathbb{R}\subseteq R$ $$S_{R}=\bigcup_{i=1}^{l}{S_{i}}_{R}$$ where  $\displaystyle {S_{i}}_{R}=\lbrace x\in R^{N} \mid  q_{i}(x)=0, p_{i,1}(x)>0,...,p_{i,m}(x)>0    \rbrace$. By Tarski--Seidenberg's principle, $S_{i}=\emptyset$ if and only if ${S_{i}}_{R}=\emptyset$, for $1\leq i \leq l$. Thus $S=\emptyset$ if and only if $S_{R}=\emptyset$. \newline Assume $S=\emptyset$ and $\widetilde{S} \neq \emptyset$. Take $\alpha \in \widetilde{S}$. From our discussion on the real spectrum we know that $\alpha$ corresponds to an $\mathbb{R}\textrm{\hyp{}}$algebra homomorphism $\phi:\mathcal{R}^{k}\left( \mathbb{R}^{N}\right)\rightarrow R$ where $R$ is a real closed extension of $\mathbb{R}$. Thus $$\phi(f)=0, \phi(f_{1})>0,..., \phi(f_{r})>0 .$$ By the substitution property (Proposition 3.10) $\phi$ is an evaluation morphism. Therefore exists $x_{0}\in R^{N}$ such that $$f_{R}(x_{0})=0, {f_{1}}_{R}(x_{0})>0,..., {f_{r}}_{R}(x_{0})>0 $$ which is a contradiction with $S_{R}=\emptyset$.    
\end{proof}
We first prove the weak Nullstellensatz for $\mathcal{R}^{k}\left( \mathbb{R}^{N}\right)$.
\begin{prop}
Let $(k,N)\in\mathbb{N}\times\mathbb{N}^{\ast}$. Let $I\subseteq\mathcal{R}^{k}\left( \mathbb{R}^{N}\right)$ be an ideal. Then $\mathcal{Z}(I)=\emptyset$ if and only if $I=\mathcal{R}^{k}\left( \mathbb{R}^{N}\right)$.
\end{prop}
\begin{proof}
By noetherianity of the regulous topology we obtain $\mathcal{Z}(I)=\mathcal{Z}(f_{1})\cap...\cap\mathcal{Z}(f_{m})$ for some finite tuple $f_{1},...,f_{m}\in I$. Taking $f:=f^{2}_{1}+...+f^{2}_{m}\in I$, we get $\mathcal{Z}(I)=\mathcal{Z}(f)$. If $\mathcal{Z}(I)=\emptyset$, then $f$ is invertible and thus $1\in I$. The converse implication is trivial.   
\end{proof}
To prove the general version of Nullstellensatz we require also the property that radical ideals are real in $\mathcal{R}^{k}\left( \mathbb{R}^{N}\right)$.
\begin{lem}
Let $(k,N)\in\mathbb{N}\times\mathbb{N}^{\ast}$. Let $I\subseteq\mathcal{R}^{k}\left( \mathbb{R}^{N}\right)$ be a radical ideal. Then $I$ is a real ideal.
\end{lem}
\begin{proof}
Suppose that $f_{1}^{2}+...+f_{m}^{2}\in I$ with  $f_{1}, ..., f_{m}\in \mathcal{R}^{k}\left( \mathbb{R}^{N}\right)$. Observe that for each $i=1, ..., m$ we have 
$$\dfrac{f_{i}^{3+k}}{f_{1}^{2}+...+f_{m}^{2}}\in\mathcal{R}^{k}\left( \mathbb{R}^{N}\right).$$
This is due to the fact that following mappings are $k\textrm{\hyp{}}$regulous: $$\mathbb{R}^{N}\ni x\mapsto(f_{1}(x), ..., f_{m}(x))\in\mathbb{R}^{m},$$
$$\mathbb{R}^{m}\ni (x_{1}, ..., x_{m})\mapsto\dfrac{x_{i}^{3+k}}{x_{1}^{2}+...+x_{m}^{2}}\in\mathbb{R}$$ and by \textit{Remark} 2.7 their composition is also $k\textrm{\hyp{}}$regulous. Thus $$f_{i}^{3+k}=\dfrac{f_{i}^{3+k}}{f_{1}^{2}+...+f_{m}^{2}}\cdot(f_{1}^{2}+...+f_{m}^{2})\in I.$$ Because $I$ is radical, we obtain $f_{i}\in I$.
\end{proof}
We still need the following version of the formal Positivstellensatz.
\begin{prop}
Let $A$ be a commutative ring (considered with $ \sum A^{2}$ cone). Let $f,g \in A$. Then the following are equivalent:
\begin{enumerate}[i)]
\item There is no homomorphism $\phi:A\rightarrow R$ into a real closed field $R$, such that $$\phi(g) \neq 0, \phi(f)=0.$$
\item There exist $h, h_{i}\in A$, $i=1,...,m$, $e\in \mathbb{N}$ such that $$\sum_{i\leq m}h^{2}_{i}+g^{2e}+fh=0.$$
\end{enumerate} 
\end{prop}
\begin{proof}
$i)\Rightarrow ii)$ We may assume $g\notin \sqrt{(f)}$, otherwise $g^{e}=fh_{1}$, for some exponent $e$ and $h_{1}\in A$, hence $g^{2e}+f\cdot(-fh_{1}^{2})=0$. So suppose $g^{e}\notin (f)$ for all $e\in \mathbb{N}$. Let $A_{1}:=A/(f)$ and $G$ be the multiplicative monoid generated by $g$. Consider $A_{2}=\overline{G}^{-1}A_{1}$ (where the bar denotes the image in $A_{1}$). Condition $i)$ implies that there is no homomorphism $\Psi :A_{2}\rightarrow R$ into a real closed field $R$. Hence, by Theorem 3.5, we can find $\delta_{1},...,\delta_{n}\in A_{2}$ such that $-1=\delta_{1}^{2}+...+\delta_{n}^{2}$. Put $\delta_{i}=\dfrac{\overline{a_{i}}}{\overline{g^{e_{i}}}}$, where $a_{i}\in A$, $e_{i}\in \mathbb{N}$. Then $-1=\dfrac{\overline{a_{1}^{2}}}{\overline{g^{2e_{1}}}}+...+\dfrac{\overline{a_{n}^{2}}}{\overline{g^{2e_{n}}}}$. Clearing denominators we obtain $\overline{g^{2e_{1}}}\cdot...\cdot \overline{g^{2e_{n}}}+\overline{p}=0$, where $p \in \sum A^{2} $. We have $g^{2(e_{1}+...+e_{n})}+p \in (f)$, and there exists $h \in A$ such that $g^{2(e_{1}+...+e_{n})}+p+fh=0$.\newline
$ii)\Rightarrow i)$ If there is a homomorphism $\phi:A\rightarrow R$ into a real closed field $R$, satisfying the conditions of $i)$, that is $\phi(g) \neq 0, \phi(f)=0$, then it must be $\phi(p+g^{2e}+fh)>0$, for any $p \in \sum A^{2}, e\in \mathbb{N}$, since $\phi(p)\geq 0, \phi(g^{2e}) > 0, \phi(fh)=0$.      
\end{proof}
\begin{rem}
The above statement in much more general case can be found in [1,~Proposition~4.4.1].
\end{rem}
We can now prove the main result of this section, the Nullstellensatz for $\mathcal{R}^{k}\left( \mathbb{R}^{N}\right)$. 
\begin{thm}
Let $(k,N)\in\mathbb{N}\times\mathbb{N}^{\ast}$. Let $I\subseteq\mathcal{R}^{k}\left( \mathbb{R}^{N}\right)$ be an ideal. Then $$\sqrt{I}=\mathcal{I}(\mathcal{Z}(I))$$.
\end{thm}
\begin{proof}
Only the "$\supseteq$" inclusion is nontrivial. Again by the noetherianity of the regulous topology, $\mathcal{Z}(I)=\mathcal{Z}(f)$, for some $f\in I$. Let $g\in \mathcal{I}(\mathcal{Z}(I))$. Obviously $\mathcal{Z}(f)\subseteq\mathcal{Z}(g)$.  By hypothesis $$\lbrace x\in \mathbb{R}^{N} \mid f(x)=0,~g(x) \neq 0 \rbrace = \emptyset.$$ Then by the Artin--Lang property (Proposition 3.11) $$\lbrace \alpha\in \mathrm{Spec_{r}} \mathcal{R}^{k}\left( \mathbb{R}^{N}\right)\mid f(\alpha)=0,~g(\alpha)\neq 0    \rbrace = \emptyset.$$
This means that there is no homomorphism $\phi:\mathcal{R}^{k}\left( \mathbb{R}^{N}\right)\rightarrow R$ into a real closed field $R$, such that $\phi(f)=0,~\phi(g) \neq 0$. By the formal Positivstellensatz (Proposition 3.14), there exist $h, h_{i}\in \mathcal{R}^{k}\left( \mathbb{R}^{N}\right)$, $i=1,...,m$, $e\in \mathbb{N}$, such that $$\sum_{i\leq m}h^{2}_{i}+g^{2e}+fh=0.$$ Hence $g$ belongs to the real radical of the ideal $(f)$. By Lemma 3.13, we get $\sqrt[r]{(f)}=\sqrt{(f)}\subseteq\sqrt{I}$.      
\end{proof}

\begin{cor}
Let $(k,N)\in\mathbb{N}\times\mathbb{N}^{\ast}$. Any maximal ideal $\mathfrak{M}$ in $\mathcal{R}^{k}\left( \mathbb{R}^{N}\right)$ is of the form $$\mathfrak{M}=\mathfrak{M}_{a}:=\lbrace f\in\mathcal{R}^{k}\left( \mathbb{R}^{N}\right)\mid f(a)=0 \rbrace$$ for an unique $a\in\mathbb{R}^{N}$. 
\end{cor}  
\begin{proof}
Let $\mathfrak{M} \subseteq \mathcal{R}^{k}\left( \mathbb{R}^{N}\right)$ be a maximal ideal. Since $\mathfrak{M}$ is proper, then $\mathcal{Z}(\mathfrak{M})\neq \emptyset$ by weak Nullstellensatz. Let $a\in \mathcal{Z}(\mathfrak{M})$, then we have $\mathfrak{M}\subseteq\mathfrak{M}_{a}$, but the inclusion cannot be proper by maximality of $\mathfrak{M}$.
\end{proof}
\begin{rem}
Of course for any $a\in \mathbb{R}^{N}$ the ideal $\mathfrak{M}_{a}$ is always maximal. Above corollary simply means that if $I$ is an ideal in the ring $\mathcal{R}^{k}\left( \mathbb{R}^{N}\right)$, then $I$ is maximal if and only if it is of the form $I=\mathfrak{M}_{a}$, for some $a\in\mathbb{R}^{N}$.     
\end{rem}

\section{Maximal ideals in the ring $\displaystyle \mathcal{R}^{k}\left( \mathbb{R}^{N}\right)$ are not finitely generated}
We begin this section with a basic lemma from which the main problem follows in the natural way. For clarity we state the following definition.

\begin{defn} Let $a \in \mathbb{R}^{N}$. For $p \in \mathbb{R}[x_{1},...,x_{N}]$, one has the unique decomposition into finite sum $\displaystyle p=\sum_{i\geq 0}p_{i}$ where $p_i$ is a homogeneous polynomial of degree $i$ with respect to $(x_{1}-a_{1}), ..., (x_{N}-a_{N})$. The smallest integer $s$ such that $p_{s}\neq0$ is order of $p$ at $a$, namely $s=\ord_{a}(p)$, in this case we call $p_{s}$ the initial form of $p$. This decomposition is called the homogeneous decomposition of $p$ at $a$, and the $p_{i}$'s are the homogeneous components.
\newline When $a=O$, the polynomial $p_{i}$ is either zero or homogeneous of degree $i$.
\end{defn}
  
\begin{lem}
Assume $(k,N)\in\mathbb{N}\times\mathbb{N}^{\ast}$, $g\in\mathcal{R}^{k}\left( \mathbb{R}^{N}\right)$. Write $g=\dfrac{p}{q}$, $q \neq 0$ on $\dom(g)$ with $p, q \in \mathbb{R}[x_{1},...,x_{N}]$. Let $p_{n}$ (resp. $q_{m}$) be the initial form of $p$ (resp.~$q$), hence $n=\ord_{O}(p)$ and $m=\ord_{O}(q)$. Then
\begin{enumerate}[i)]
\item $n\geq m$ and $m$ is even. Moreover $g(O)=0$ if and only if $n>m$.
\item We can write $g=h+G$ where $h\in\mathcal{R}^{k}\left( \mathbb{R}^{N}\right)$ is $k\textrm{\hyp{}}$flat at $O$ and $G \in \mathbb{R}[x_{1},...,x_{N}]$ such that $\deg G\leq k$.  
\item If $g$ is $k\textrm{\hyp{}}$flat at $O$, then $n=\deg p_{n}> \deg q_{m}+k=m+k$. 
\end{enumerate}  
\end{lem}
\begin{proof} Consider a generic line $l$ through $O$, for which $q_{m}\mid_{l}\neq 0$, parametrized by $$t\mapsto\left(t, \alpha_{1}t, \alpha_{2}t, ..., \alpha_{N-1}t\right)$$
\begin{enumerate}[i)]
\item The sign of $q$ on $l$ in a neighbourhood of $O$ depends only on its initial form $q_{m}$. Suppose that $m$ is odd, then passing through $O$, $q_{m}$ changes sign on $l$, so does $q$. But it follows from [1, Theorem 4.5.1] that $\textrm{codim}_{\mathbb{R}^{N}}\mathcal{Z}(q)=1$ which is a contradiction with Proposition 2.5.\newline Next, observe that $g(O)$ depends only on $ \displaystyle t^{n-m}\dfrac{p_{n}\left(1, \alpha_{1}, \alpha_{2}, ..., \alpha_{N-1}\right)}{q_{m}\left(1, \alpha_{1}, \alpha_{2}, ..., \alpha_{N-1}\right)}$ as $t\rightarrow0$.
\item It suffices to consider the Taylor polynomial of degree $k$, $G(x):= T_{O}^{k}g(x)$ of $g(x)$ at $O$ and then take $h:=g-G$.
\item Because $\restr{g}{l}$ is also $k\textrm{\hyp{}}$flat at $O$, then $\ord_{O}(\restr{p}{l})>m+k$ on a generic line $l$. But also on a generic line $l$ we have $n=\ord_{O}(p)=\ord_{O}(\restr{p}{l})$. 
\end{enumerate}
\end{proof}
It is well known that the ring $\mathcal{R}^{k}\left( \mathbb{R}^{N}\right)$, $N\geq2$, is not noetherian (cf. [2, Proposition~4.16]). Unfortunately the method used in [2] does not lead to Theorem 4.3. Using different approach we prove below that in fact maximal ideals of $\mathcal{R}^{k}\left( \mathbb{R}^{N}\right)$ are not finitely generated.

\begin{thm}
Let $(k,N)\in\mathbb{N}\times\mathbb{N}^{\ast}$. The maximal ideals of the ring of $k\textrm{\hyp{}}$regulous functions $\mathcal{R}^{k}\left( \mathbb{R}^{N}\right)$ are not finitely generated for $N\geq2$.   
\end{thm}
\begin{proof}
Fix $N\geq 2$, $k\in\mathbb{N}$. Let $\mathfrak{M}$ be a maximal ideal of $\mathcal{R}^{k}\left( \mathbb{R}^{N}\right)$. By Corollary 3.17, $\mathfrak{M}$ is of the form $\mathfrak{M}=\mathfrak{M}_{a}$ for a point $a\in\mathbb{R}^{N}$. Clearly, we may assume that $a=O$.   

 Suppose the contrary and assume that there exist a finite set of generators $\widetilde{f}_{1},...,\widetilde{f}_{r}\in\mathfrak{M}_{O}$. We can write $\widetilde{f}_{i}=f_{i}+F_{i}$ with some $f_{i}$ $k\textrm{\hyp{}}$flat at $O$, and $F_{i}$ a polynomial with $\textrm{ord}_{O}F_{i}\geq 1$. For each $i=1,...,r$ write $f_{i}$ in the form 
$f_{i}(x)=\dfrac{p_{i}(x)}{q_{i}(x)}$, $q_{i}(x)\neq0$ for all $x\in \mathrm{dom}(f_{i})$, where $p_{i}, q_{i}$ are polynomials in $ \mathbb{R}[x_{1},...,x_{N}]$. Denote by $p_{i, n_{i}}, q_{i, m_{i}}$ the initial form of $p_{i}, q_{i}$ respectively. Then $f_{i}=\dfrac{p_{i, n_{i}}+hot_{p_{i}}}{q_{i, m_{i}}+hot_{q_{i}}}$, where $hot_{p_{i}}, hot_{q_{i}}$ are sums of homogeneous components of $p_{i}, q_{i}$ of degrees greater than $n_{i}, m_{i}$ respectively. 
 Define a $k\textrm{\hyp{}}$regulous function $f$ by putting $$f\left( x_{1}, x_{2}, ...,x_{N} \right) = \dfrac{x_{1}^{3+k}}{a_{1}x_{1}^{2}+a_{2}x_{2}^{2}+...+a_{N}x_{N}^{2}}$$ away from the origin, where we choose $a_{1}, a_{2},..., a_{N}$ to be positive real numbers such that none of $q_{i,m_{i}}$ is divisible by $a_{1}x_{1}^{2}+a_{2}x_{2}^{2}+...+a_{N}x_{N}^{2}$. \newline Clearly $f\in \mathfrak{M}_{O}$ and therefore there are $\widetilde{g}_{1},...,\widetilde{g}_{r}\in\mathcal{R}^{k}\left( \mathbb{R}^{N}\right)$, such that: $$f=\widetilde{g}_{1}\widetilde{f}_{1}+\widetilde{g}_{2}\widetilde{f}_{2}+...+\widetilde{g}_{r}\widetilde{f}_{r}.$$ Again write $\widetilde{g}_{i}=g_{i}+G_{i}$ with some $g_{i}$ $k\textrm{\hyp{}}$flat at $O$, and $G_{i}$ a polynomial. Then:
\begin{equation} 
f=\sum_{i\leq r}\left (g_{i}+G_{i}\right )\left (f_{i}+F_{i}  \right ) 
\end{equation} 
and after routine calculation we obtain $\displaystyle \textrm{ord}_{O} \sum_{i\leq r}G_{i}F_{i}\geq k+1$.       
\newline Let us now consider a generic line $\textit{l}$ in $ \mathbb{R}^{N}$ parametrized by $t\mapsto\left(t, \alpha_{1}t, \alpha_{2}t, ..., \alpha_{N-1}t\right)$. By equality (4.4), on a given line we have: 
$$\dfrac{t^{k+1}}{a_{1}+a_{2}\alpha_{1}^{2}+...+a_{N}\alpha _{N-1}^{2}}=\sum_{i\leq r}\restr{(g_{i}f_{i})}{\textit{l}}+\sum_{i\leq r}\restr{(g_{i}F_{i})}{\textit{l}}+\sum_{i\leq r}\restr{(G_{i}f_{i})}{\textit{l}}+\sum_{i\leq r}\restr{(G_{i}F_{i})}{\textit{l}}$$
\newline Next consider the situation at the origin, that is when $t\rightarrow 0$. Because $f_{i}$ are $k\textrm{\hyp{}}$flat at $O$, from lemma 4.2 it follows, that either $n_{i}=m_{i}+k+1$ and then $$\displaystyle \lim_{t\rightarrow 0}\dfrac{f_{i}\left(t, \alpha_{1}t, \alpha_{2}t, ..., \alpha_{N-1}t\right)}{t^{k+1}}=\dfrac{p_{i, n_{i}}\left(1, \alpha_{1}, \alpha_{2}, ..., \alpha_{N-1}\right)}{q_{i, m_{i}}\left(1, \alpha_{1}, \alpha_{2}, ..., \alpha_{N-1}\right)}$$\newline or $n_{i}> m_{i}+k+1$ and then the above limit is zero. \newline Since $\displaystyle \lim_{t\rightarrow 0}\dfrac{g_{i}\left(t, \alpha_{1}t, \alpha_{2}t, ..., \alpha_{N-1}t\right)}{t^{k}}=0$ and $\textrm{ord}_{O}F_{i}\geq 1$ we get $$\displaystyle \lim_{t\rightarrow 0}\dfrac{g_{i}\left(t, \alpha_{1}t, \alpha_{2}t, ..., \alpha_{N-1}t\right)F_{i}\left(t, \alpha_{1}t, \alpha_{2}t, ..., \alpha_{N-1}t\right)}{t^{k+1}}=0.$$ \newline
Let $\displaystyle H=\sum_{i\leq r}G_{i}F_{i}$. We have already noticed that  $\displaystyle \textrm{ord}_{O}H\geq k+1$, thus $$\lim_{t\rightarrow 0}\dfrac{H\left(t, \alpha_{1}t, \alpha_{2}t, ..., \alpha_{N-1}t\right)}{t^{k+1}}=h\left(\alpha_{1}, \alpha_{2}, ..., \alpha_{N-1}\right)$$ where $h$ is some polynomial in $ \mathbb{R}[x_{1},...,x_{N-1}]$. 
\newline Summing up, after dividing by $t^{k+1}$ and taking limits, we obtain
$$\dfrac{1}{a_{1}+a_{2}\alpha_{1}^{2}+...+a_{N}\alpha _{N-1}^{2}}=\sum_{i\leq r}G_{i}(O) \lim_{t\rightarrow 0}\dfrac{f_{i}\left(t, \alpha_{1}t, ..., \alpha_{N-1}t\right)}{t^{k+1}}+h\left(\alpha_{1}, \alpha_{2}, ..., \alpha_{N-1}\right).$$ But we have already seen that $\displaystyle \lim_{t\rightarrow 0}\dfrac{f_{i}\left(t, ..., \alpha_{N-1}t\right)}{t^{k+1}}=\dfrac{p_{i, n_{i}}\left(1, \alpha_{1}, \alpha_{2}, ..., \alpha_{N-1}\right)}{q_{i, m_{i}}\left(1, \alpha_{1}, \alpha_{2}, ..., \alpha_{N-1}\right)}$ or $0$. Hence the product $\displaystyle \prod_{i\leq r}q_{i, m_{i}}\left(1, x_{2}, x_{3}, ..., x_{N}\right)$ is divisible by $a_{1}+a_{2}x_{2}^{2}+...+a_{N}x_{N}^{2}$ in the ring $ \mathbb{R}[x_{2},...,x_{N}]$. \newline Then obviously $\displaystyle \sum_{i\leq r}m_{i} \geq 2$ and it follows that $a_{1}x_{1}^{2}+a_{2}x_{2}^{2}+...+a_{N}x_{N}^{2}$ divides the product $$\displaystyle \prod_{i\leq r}x_{1}^{m_{i}}q_{i, m_{i}}\left(1, \dfrac{x_{2}}{x_{1}}, ..., \dfrac{x_{N}}{x_{1}}\right)=\prod_{i\leq r}q_{i, m_{i}}\left(x_{1}, x_{2}, ..., x_{N}\right)$$ in the ring $ \mathbb{R}[x_{1},...,x_{N}]$, which is a contradiction.   
\end{proof}
We generalize the above result to an arbitrary nonsingular, real affine algebraic variety of dimension $d\geq2$. We will need the following result from the paper [4] which allows us to extend a regulous function from a smooth, closed subvariety to the ambient space.  

\begin{thm}
Let $X$ be an irreducible, smooth, real algebraic set $X \subset \mathbb{R}^{N}$. Let $\mathcal{I}(X)$ be the ideal 
of regulous functions on $\mathbb{R}^{N}$ vanishing on $X$. Then $$\mathcal{R}^{0}\left( X\right)\cong\mathcal{R}^{0}\left( \mathbb{R}^{N}\right)/\mathcal{I}(X)$$
\end{thm}
\begin{proof}
cf. [4, Proposition 8 and Theorem 10].
\end{proof}
\noindent Therefore the regulous functions on $X$ are precisely the restrictions of the regulous functions defined on the ambient space $\mathbb{R}^{N}$. 
\begin{defn}
Let $X \subset \mathbb{R}^{N}$ be an irreducible, smooth, real algebraic set. We define the $k\textrm{\hyp{}}$regulous topology on $X$ as a topology induced by the family of $k\textrm{\hyp{}}$regulous closed sets, i.e. the sets of the form $\mathcal{Z}(F)$, where $F\subseteq \mathcal{R}^{k}\left( X\right)$.  
\end{defn}
As an immediate consequence of the Theorem 4.5 we obtain the following result. 
\begin{prop}
The $k\textrm{\hyp{}}$regulous topology on an irreducible, smooth, real algebraic set $X \subset \mathbb{R}^{N}$ is noetherian for all $k\in\mathbb{N}$.
\end{prop} 
\begin{proof}
Because any chain of $k\textrm{\hyp{}}$regulous closed sets is also a chain of regulous closed sets, it is enaugh to prove the proposition for $k=0$. Let $V\subset X$ be regulous closed. Then $V=\mathcal{Z}(F)$ for some $F\subset\mathcal{R}^{k}\left( X\right)$. Define $\tilde{F}:=\lbrace \tilde{f} \in \mathcal{R}^{0}\left( \mathbb{R}^{N}\right): \tilde{f}|_{X} \in F \rbrace $ and take $W:=\mathcal{Z}(\tilde{F})$ a regulous closed set in $\mathbb{R}^{N}$. By Theorem 4.5, $V=W\cap X$, and by noetherianity of regulous topology in $\mathbb{R}^{N}$ we know that $W=\mathcal{Z}(\tilde{f}_{1},...,\tilde{f}_{m})$ for a finite tuple $\tilde{f}_{1},...,\tilde{f}_{m}\in \tilde{F}$. Let $f_{i}:=\tilde{f}_{i}\mid _{X} \in F$, for $i=1,...,m$. We conclude that $V=\mathcal{Z}(f_{1},...,f_{m})$.
\end{proof}

Noetherianity of $k\textrm{\hyp{}}$regulous topology allows us to obtain the weak Nullstellensatz for the ring of $k\textrm{\hyp{}}$regulous functions on a smooth, real affine algebraic variety. 

\begin{prop}[Weak Nullstellensatz]
Let $X$ be an irreducible, smooth, real algebraic set $X \subset \mathbb{R}^{N}$, $k\in \mathbb{N}$. Let $I\subseteq\mathcal{R}^{k}\left(X\right)$ be an ideal. Then $\mathcal{Z}(I)=\emptyset$ if and only if $I=\mathcal{R}^{k}\left( X\right)$.
\end{prop}
\begin{proof}
The same as in the Proposition 3.12.
\end{proof}

\begin{cor}
Let $X$ be an irreducible, smooth, real algebraic set $X \subset \mathbb{R}^{N}$, $k\in \mathbb{N}$. Any maximal ideal $\mathfrak{M}$ in $\mathcal{R}^{k}\left( X\right)$ is of the form $$\mathfrak{M}=\mathfrak{M}_{a}:=\lbrace f\in\mathcal{R}^{k}\left( X\right)\mid f(a)=0 \rbrace$$ for an unique $a\in X$. 
\end{cor}  
\begin{proof}
The same as in the Corollary 3.17.
\end{proof}

We can now generalize the Theorem 4.3 to the case of smooth, real affine algebraic varieties. 

\begin{thm}
Let $k\in\mathbb{N}$. The maximal ideals of the ring of $k\textrm{\hyp{}}$regulous functions $\mathcal{R}^{k}\left( X\right)$ on a smooth, real affine algebraic variety $X \subset \mathbb{R}^{N}$ of dimension $d$, are not finitely generated for $d\geq2$. In particular the ring $\mathcal{R}^{k}\left( X\right)$ is not noetherian, provided that $\mathrm{dim}X\geq 2$.
\end{thm}

\begin{proof}
Assume that $X \subset \mathbb{R}^{N} $ is an irreducible, smooth, real algebraic set of dimension $d\geq2$ and let $\mathfrak{M}$ be a maximal ideal of $\mathcal{R}^{k}\left( X\right)$. By Corollary 4.9, $\mathfrak{M}$ is of the form $\mathfrak{M}=\mathfrak{M}_{a}$ for a point $a\in X$. Clearly, we may assume that $a=O$. Denote by $T_{O}(X)$ the Zariski tangent space of $X$ at $O$ (cf. [1, Definition 3.3.3]). Since $X$ is smooth, w.l.o.g. we may assume that $T_{O}(X)= \mathbb{R}^{d}\times {\lbrace0\rbrace}^{N-d}$. \newline \indent Denote by $\mathrm{pr}:\mathbb{R}^{N}\rightarrow T_{O}(X)$ the orthogonal projection. From [1, Propositions 3.3.10, 2.9.7 and 8.1.8] there exists an euclidean open neighbourhood $U$ of $O$ in $\mathbb{R}^{N}$ such that $\mathrm{pr}|_{X \cap U}$ is a Nash (analytic) diffeomorphism onto its image $\mathrm{pr}|_{X \cap U}(X \cap U)\subset T_{O}(X)$, where $\mathrm{pr}(O)=O$. Let $V\subset \mathbb{R}^{d}$ be an open neighbourhood of $O$ such that $\iota (V)= \mathrm{pr}|_{X \cap U}(X \cap U)$, where $\iota :\mathbb{R}^{d}\rightarrow \mathbb{R}^{N}$ is the natural embedding. Write $\mathrm{pr}^{-1}\circ \iota:V \rightarrow X \cap U $ as $$(x_{1},...,x_{d})\mapsto (x_{1},...,x_{d},N_{1}(x_{1},...,x_{d}),...,N_{N-d}(x_{1},...,x_{d}))$$ where $N_{j}$ are analytic functions (precisely the Nash functions) such that $\dfrac{\partial N_{j}}{\partial x_{i}}(O)=0$, for $j=1,...,N-d$ and $i=1,...,d$. 
\newline \indent Consider a $k\textrm{\hyp{}}$regulous function $f \in \mathcal{R}^{k}\left( X\right)$. Write $f=\dfrac{p}{q}$, where $p,q \in \mathbb{R}[x_{1},...,x_{N}]$ with $q\neq 0$ on $\dom{(f)}$. Then $p\circ \mathrm{pr}^{-1}\circ \iota$ and $q\circ \mathrm{pr}^{-1}\circ \iota$ are analytic functions on $V$, and $q\circ \mathrm{pr}^{-1}\circ \iota$ is nonvanishing everywhere on $V$. Let $\hat{p}, \hat{q} \in \mathbb{R}[x_{1},...,x_{d}]$ be the initial forms of the Taylor series expansion around the origin $O \in V$ of $p\circ \mathrm{pr}^{-1}\circ \iota$ and $q\circ \mathrm{pr}^{-1}\circ \iota$ respectively. \newline Assume that $f\circ \mathrm{pr}^{-1}\circ \iota$ is $k\textrm{\hyp{}}$flat at $O$. We claim that $\deg\hat{p}>\deg\hat{q}+k$. \newline Take a parametrized line $l$ given by $\left(\alpha_{1}t, \alpha_{2}t,..., \alpha_{d}t\right)$ such that $\hat{q}$ does not vanish on $l$. Then $\restr{f\circ \mathrm{pr}^{-1}\circ \iota}{l\cap V}$ is also $k\textrm{\hyp{}}$flat at $0$. Moreover, just as in the Lemma 4.2, observe that $f(O)$ depends only on $t^{\deg\hat{p}-\deg\hat{q}}\dfrac{\hat{p}\left(\alpha_{1}, \alpha_{2},..., \alpha_{d}\right)}{\hat{q}\left(\alpha_{1}, \alpha_{2},..., \alpha_{d}\right)}$ as $t\rightarrow 0$. But on the generic line $l$ we get that $\deg\hat{p}=\deg(\restr{\hat{p}}{l})>\deg(\restr{\hat{q}}{l})+k=\deg\hat{q}+k$, therefore we have concluded the above claim.   
\newline \indent Suppose that $\mathfrak{M}_{O}$ is finitely generated by $\widetilde{f}_{1},...,\widetilde{f}_{r}\in\mathfrak{M}_{O}$. Write $\widetilde{f}_{i}=f_{i}+F_{i}$, where $f_{i}=\widetilde{f}_{i}-F_{i}$ and $F_{i}=T_{O}^{k}(\widetilde{f}_{i}\circ \mathrm{pr}^{-1}\circ \iota )\in \mathbb{R}[x_{1},...,x_{d}]$ is the Taylor polynomial of degree $k$ at $O$ of $\widetilde{f}_{i}\circ \mathrm{pr}^{-1}\circ \iota$ which is of class $\mathcal{C}^{k}$ near $O$ in $\mathbb{R}^{d}$, with $\textrm{ord}_{O}F_{i}\geq 1$. Then $f_{i}\in \mathcal{R}^{k}\left( X\right)$ and $f_{i}\circ \mathrm{pr}^{-1}\circ \iota$ is $k\textrm{\hyp{}}$flat at $O$. Define a $k\textrm{\hyp{}}$regulous (also on $\mathbb{R}^{N}$) function $f$ on $X$ by putting $$f\left( x_{1}, x_{2}, ...,x_{N} \right) = \dfrac{x_{1}^{3+k}}{a_{1}x_{1}^{2}+a_{2}x_{2}^{2}+...+a_{N}x_{N}^{2}}$$ away from the origin, where we choose $a_{1}, a_{2},..., a_{N}$ to be positive real numbers such that none of $\hat{q}_{i}$ is divisible by $a_{1}x_{1}^{2}+a_{2}x_{2}^{2}+...+a_{d}x_{d}^{2}$, where $\hat{p}_{i}$, $\hat{q}_{i}$ are the initial forms of the Taylor series expansion around the origin of $p_{i}\circ \mathrm{pr}^{-1}\circ \iota$ and $q_{i}\circ \mathrm{pr}^{-1}\circ \iota $ respectively, where $f_{i}=\dfrac{p_{i}}{q_{i}}$, $p_{i},q_{i} \in \mathbb{R}[x_{1},...,x_{N}]$ for $i=1,...,r$. Clearly $f$ is $k\textrm{\hyp{}}$flat at $O$. There exists $\widetilde{g}_{1},...,\widetilde{g}_{r}\in\mathcal{R}^{k}\left( X\right)$ such that 
\begin{equation}
f=\widetilde{g}_{1}\widetilde{f}_{1}+\widetilde{g}_{2}\widetilde{f}_{2}+...+\widetilde{g}_{r}\widetilde{f}_{r}
\end{equation}   
 Write $\widetilde{g}_{i}=g_{i}+G_{i}$ with $G_{i}=T_{O}^{k}(\widetilde{g}_{i}\circ \mathrm{pr}^{-1}\circ \iota )\in \mathbb{R}[x_{1},...,x_{d}]$ and $g_{i}\in \mathcal{R}^{k}\left( X\right)$ such that $g_{i}\circ \mathrm{pr}^{-1}\circ \iota$ is $k\textrm{\hyp{}}$flat at $O$ in $\mathbb{R}^{d}$.
 \newline By (4.11) we know that $$f=\sum_{i\leq r}g_{i}f_{i}+\sum_{i\leq r}g_{i}F_{i}+\sum_{i\leq r}G_{i}f_{i}+\sum_{i\leq r}G_{i}F_{i}$$ Note that $ \displaystyle \ord_{O}\sum_{i\leq r}G_{i}F_{i}\geq k+1$. Let us now consider a generic line $\textit{l}$ in $ \mathbb{R}^{d}$ parametrized by $t\mapsto\left(\alpha_{1}t, \alpha_{2}t,..., \alpha_{d}t\right)$. Because of $k\textrm{\hyp{}}$flatness of $f_{i}\circ \mathrm{pr}^{-1}\circ \iota$, $g_{i}\circ \mathrm{pr}^{-1}\circ \iota$ and $\textrm{ord}_{O}F_{i}\geq 1$, we know that both of the following limits vanish
 $$
\lim_{t\rightarrow 0}\dfrac{g_{i}f_{i}\circ \mathrm{pr}^{-1}\circ \iota\left(\alpha_{1}t,..., \alpha_{d}t\right)}{t^{k+1}}=0$$ $$\lim_{t\rightarrow 0}\dfrac{g_{i}F_{i}\left(\alpha_{1}t,..., \alpha_{d}t\right)}{t^{k+1}}=0
 $$ Next, observe that $$\lim_{t\rightarrow 0}\dfrac{f_{i}\circ \mathrm{pr}^{-1}\circ \iota\left(\alpha_{1}t, \alpha_{2}t,..., \alpha_{d}t\right)}{t^{k+1}}=\dfrac{\hat{p}_{i}\left(\alpha_{1}, \alpha_{2},..., \alpha_{d}\right)}{\hat{q}_{i}\left(\alpha_{1}, \alpha_{2},..., \alpha_{d}\right)}$$ if $\deg\hat{p}_{i}=\deg\hat{q}_{i}+k+1$, or the above limit is $0$ if $\deg\hat{p}_{i}>\deg\hat{q}_{i}+k+1$ for $i=1,...,r$.
Let $\displaystyle H=\sum_{i\leq r}G_{i}F_{i}$. We have already noticed that  $\displaystyle \textrm{ord}_{O}H\geq k+1$, thus $$\lim_{t\rightarrow 0}\dfrac{H\left(\alpha_{1}t, \alpha_{2}t,..., \alpha_{d}t\right)}{t^{k+1}}=h\left(\alpha_{1}, \alpha_{2}, ..., \alpha_{d}\right)$$ where $h$ is some polynomial in $ \mathbb{R}[x_{1},...,x_{d}]$. Because $\displaystyle \lim_{t\rightarrow 0}\dfrac{N_{j}\left(\alpha_{1}t, \alpha_{2}t,..., \alpha_{d}t\right)}{t}=0$ for $j=1,...,N-d$, we get that \begin{equation}
\lim_{t\rightarrow 0}
\dfrac{f\circ \mathrm{pr}^{-1}\circ \iota\left(\alpha_{1}t, \alpha_{2}t,..., \alpha_{d}t\right)}{t^{k+1}}= \dfrac{\alpha^{3+k}_{1}}{a_{1}\alpha_{1}^{2}+a_{2}\alpha_{2}^{2}+...+a_{d}\alpha_{d}^{2}}
\end{equation} Let $I\subseteq \lbrace 1,...,r \rbrace$ be the set of indices $i\leq r$ for which $\deg\hat{p}_{i}=\deg\hat{q}_{i}+k+1$. We obtain that (4.12) is equal to $$\sum_{i\in I}G_{i}(O)\dfrac{\hat{p}_{i}(\alpha_{1}, \alpha_{2},..., \alpha_{d})}{\hat{q}_{i}(\alpha_{1}, \alpha_{2},..., \alpha_{d})}+h\left(\alpha_{1}, \alpha_{2}, ..., \alpha_{d}\right)$$ 

\noindent Then, after routine calculation, just like in the affine case, we finally get that 
$$a_{1}x_{1}^{2}+a_{2}x_{2}^{2}+...+a_{d}x_{d}^{2}|x_{1}^{3+k}\prod_{i\leq r}\hat{q}_{i}(x_{1},...,x_{d})$$ in the ring $\mathbb{R}[x_{1},...,x_{d}]$ - which contradicts our assumptions.
\end{proof}

\subsection*{Acknowledgments}  It is a pleasure to thank my advisor, Professor Krzysztof Nowak, for suggesting me the topic of the work, especially pointing out problems and influential discussions.

\textsc{\small Instytut Matematyki, Wydzia\l~ Matematyki i Informatyki, Uniwersytet Jagiello\'nski, ul. \L{}ojasiewicza 6, 30-348 Krak\'ow, Poland}\\

\textit{E-mail address}, A.~Czarnecki: \texttt{aleksander.czarnecki@doctoral.uj.edu.pl}
\end{document}